\documentclass[a4paper,11pt]{amsart}

\usepackage{amssymb}
\usepackage{mathrsfs}
\usepackage{amsmath,amssymb,amsthm,latexsym,amscd,mathrsfs}
\usepackage{indentfirst}
\usepackage{stmaryrd}
\usepackage{graphicx}
\usepackage{extarrows}
\usepackage{multirow}
\usepackage{latexsym}
\usepackage{amsfonts}
\usepackage{color}
\usepackage{pictexwd,dcpic}
\usepackage{graphicx}
\usepackage{psfrag}
\usepackage{hyperref}
\usepackage{comment}

\numberwithin{equation}{section} \theoremstyle{plain}
\newtheorem{thm}{Theorem}[section]

\newtheorem{lem}[thm]{Lemma}
\newtheorem{cor}[thm]{Corollary}
\newtheorem{rem}[thm]{Remark}

\newtheorem{ack}{Acknowledgements}   
 \makeatletter

\def\<{\langle}
\def\>{\rangle}
\def\({\left(}
\def\){\right)}
\def\[{\left[}
\def\]{\right]}
\def\tr{\mathop{\text{tr}}}
\def\diag{\mathop{\text{diag}}}
\def\rank{\mathop{\text{rank}}}
\def\Re{\mathop{\text{Re}}}

\makeatother

\title{Some generalizations of the DDVV  and  BW inequalities}

\author[J.Q. Ge]{Jianquan Ge}
\address{School of Mathematical Sciences, Laboratory of Mathematics and Complex Systems, Beijing Normal University, Beijing 100875, P.R. CHINA.}
\email{jqge@bnu.edu.cn}

\author[F.G. Li]{Fagui Li}
\email{faguili@mail.bnu.edu.cn}

\author[Y. Zhou]{Yi Zhou}
\email{zhou\_yi@mail.bnu.edu.cn}

\date{}

\thanks{The first author is partially supported by Beijing Natural Science Foundation (No. Z190003), NSFC (No. 11522103,
11331002).}

\begin{document}
\maketitle

\begin{abstract}
In this paper we generalize the known DDVV-type inequalities for real (skew-)symmetric and complex (skew-)Hermitian matrices to arbitrary real, complex and quaternionic matrices.
Inspired by the Erd\H{o}s-Mordell inequality, we establish the DDVV-type inequalities for matrices in the subspaces spanned by a Clifford system or a Clifford algebra.
We also generalize the B\"{o}ttcher-Wenzel inequality to quaternionic matrices.
\end{abstract}

\section{Introduction}\label{sec-int}
The DDVV-type inequality originates from the normal scalar curvature conjecture in submainfold geometry.
In 1999, De Smet, Dillen, Verstraelen and Vrancken \cite{DDVV99} proposed the normal scalar curvature conjecture (DDVV conjecture):
Let $M^n\rightarrow N^{n+m}(\kappa)$ be an isometric immersed $n$-dimensional submanifold in the real space form with constant sectional curvature $\kappa$.
Then there is the pointwise inequality $$\rho+\rho^\bot\leq\|H\|^2+\kappa, $$
where $\rho$ is the scalar curvature (intrinsic invariant), $H$ is the mean curvature vector field and $\rho^\bot$ is the normal scalar curvature (extrinsic invariants).
F. Dillen, J. Fastenakels and J. Veken \cite{DFV07} transformed this conjecture into an equivalent algebraic version (DDVV inequality):
Let $B_1, \cdots, B_m$ be $n\times n$ real symmetric matrices, then
$$\sum^m_{r,s=1}\|[B_r,B_s]\|^2\leq\(\sum^m_{r=1}\|B_r\|^2\)^2,$$
where $[A,B]=AB-BA$ is the commutator and $\|B\|^2=\tr(BB^t)$ is the squared Frobenius norm.
There were many researches on the DDVV conjecture (cf. \cite{DHTV04, Lu07, CL08, GT11} etc.);
Lu \cite{Lu11} and Ge-Tang \cite{GT08} finally proved the DDVV inequality independently and differently.
Submanifolds achieving the equality everywhere are called Wintgen ideal submanifolds, which are not classified so far (cf. \cite{CL08, DT09, XLMW14, GTY16}). Besides of this original geometric background, the DDVV inequality has also many important applications, for example, in deriving Simons-type inequalities and pinching results for the second fundamental form in submanifold geometry (cf. \cite{Lu11, LW16, GX12}).

A DDVV-type inequality is an optimal estimate of how big the pairwise commutators between a series of certain $n\times n$ matrices $B_1, \cdots, B_m$ can be:
\begin{equation}\label{DDVVtypeineq}
\sum^m_{r,s=1}\|[B_r,B_s]\|^2\leq c\(\sum^m_{r=1}\|B_r\|^2\)^2,
\end{equation}
which retains the form of the DDVV inequality.
We are interested in the optimal smallest constant $c$ so that (\ref{DDVVtypeineq}) stays valid for all matrices in the regarded class. Ge \cite{Ge14} proved the DDVV-type inequality for real skew-symmetric matrices, and applied it to get a Simons-type inequality for Yang-Mills fields in Riemannian submersion geometry. Ge-Xu-You-Zhou \cite{GXYZ17} extended the DDVV-type inequalities from real symmetric and skew-symmetric matrices to Hermitian and skew-Hermitian matrices. In Section \ref{sec2} of this paper, by using the DDVV-type inequalities for real symmetric and skew-symmetric matrices,  we will firstly give a much simpler proof of the DDVV-type inequality for (skew-)Hermitian matrices. The idea is to divide the Hermitian matrices into real symmetric matrices as real part and skew-symmetric matrices as imaginary part. It turns out that for complex (skew-)symmetric matrices the optimal constant $c$ is the same as in the real case. This new technique will also be used to prove the DDVV-type inequality for general complex matrices (and thus general real matrices) by dividing complex matrices into Hermitian matrices and skew-Hermitian matrices. The optimal constant $c$ in this case is $\frac{4}{3}$ when $m\geq3$. We summarize the DDVV-type inequalities mentioned above by indicating the optimal constant $c$ with respect to the types of the matrices in the following Table \ref{table-1} (see \cite{GT08, Ge14, GXYZ17},  Theorem \ref{TcDDVV} and Corollaries \ref{cor-sym}, \ref{cor-skew} in Section \ref{sec2} for equality conditions).
\begin{table}[h]
\caption{The optimal constant $c$ of DDVV-type inequalities ($m\geq3$)} \label{table-1}
\centering
\begin{tabular}{|c|c|c|}
\hline
c &real &complex\\
\hline
symmetric & $1$ &$1$\\
\hline
skew-symmetric &$\frac13(n=3), \frac23(n\geq4)$ &$\frac13(n=3), \frac23(n\geq4)$ \\
\hline
Hermitian & --  &$\frac43$\\
\hline
skew-Hermitian & -- &$\frac43$\\
\hline
general &$\frac43$ &$\frac43$\\
\hline
\end{tabular}
\end{table}

When $m=2$, the DDVV inequality can be derived from the B\"{o}ttcher-Wenzel inequality (BW inequality for short):
Let $X, Y$ be arbitrary real matrices, then $$\|\[X,Y\]\|^2\leq 2\|X\|^2\|Y\|^2. $$
The BW inequality was conjectured by B\"{o}ttcher-Wenzel \cite{BW05} and proved by B\"{o}ttcher-Wenzel \cite{BW08} ,
Vong-Jin \cite{VJ08} and Lu \cite{Lu11, Lu12} in many different ways.
B\"{o}ttcher-Wenzel \cite{BW08} also extended the BW inequality from real matrices to complex matrices,
and Cheng-Vong-Wenzel \cite{CVW10} obtained the equality condition.
Lu-Wenzel \cite{LW17} summarized the relevant results and conjectured a unified generalization of the DDVV inequality and the BW inequality. See also \cite{GLLZ20} for equivalent characterizations of the Lu-Wenzel conjecture and some partial results.
In Section \ref{sec3}, we generalize the BW inequality and the DDVV inequality from complex matrices to quaternionic matrices by using the same technique mentioned before: dividing quaternionic matrices into complex matrices by the standard homomorphism. We summarize these inequalities for general real, complex and quaternionic matrices in the following Table \ref{Table2}
(see \cite{CVW10}, Section \ref{sec2} and Theorems \ref{TqBW}, \ref{TqBW2}, \ref{TqDDVV} in Section \ref{sec3} for equality conditions).
\begin{table}[h!]
\caption{The optimal constant $c$ of DDVV-type and BW-type inequalities}\label{Table2}
\centering
\begin{tabular}{|c|c|c|c|}
\hline
c &real &complex &quaternionic\\
\hline
DDVV$(m\geq3)$ &$\frac43$ &$\frac43$ &$\frac83$\\
\hline
DDVV$(m=2)$ &1 &1 &2\\
\hline
BW &2 &2 &4\\
\hline
\end{tabular}
\end{table}

In 2016 Z. Lu pointed out to us that the DDVV-type inequality for real skew-symmetric matrices is in some sense a generalization of the Erd\H{o}s-Mordell inequality (cf. \cite{E35, MB37, Ba58, Le61, DNP16, MM17}, etc.): Let $P$ be an interior point of a triangle $\triangle ABC$ and $P_A,P_B,P_C$ be its projection points to the three edges, then the sum of the distances to the three vertices is no less than twice of the sum of the distances to the three edges, i.e., $$2(PP_A+PP_B+PP_C)\leq PA+PB+PC.$$
The Erd\H{o}s-Mordell inequality is implied by the following inequality for $n=3$ (cf. \cite{GS05}):
Let $a_1,\cdots,a_n\in \mathbb{R}$ and let $\alpha_1+\cdots+\alpha_n=\pi$, then
$$\sec\frac{\pi}{n}~(a_1a_2\cos\alpha_1+\cdots+a_{n-1}a_n\cos\alpha_{n-1}+a_na_1\cos\alpha_n)\leq a_1^2+\cdots+a_n^2.$$
 Considering the natural isomorphism between $\mathbb{R}^3$ and the space $o(3)$ of $3\times3$ real skew-symmetric matrices, the Erd\H{o}s-Mordell inequality is exactly the DDVV-type inequality for real skew-symmetric matrices restricted to a $2$-dimensional subspace of $o(3)$ which has a smaller optimal constant $c=\frac{1}{4}$. In fact, taking $x_1=\overrightarrow{PA},x_2=\overrightarrow{PB},x_3=\overrightarrow{PC}\in\mathbb{R}^2\subset \mathbb{R}^3$ (corresponding to $X_1,X_2,X_3\in V\subset o(3)$ where $V$ is a $2$-dimensional subspace) with norms $a_1,a_2,a_3$ and pairwise angles $\frac{\pi}{2}+\frac{\alpha_3}{2},\frac{\pi}{2}+\frac{\alpha_2}{2},\frac{\pi}{2}+\frac{\alpha_1}{2}$, the Erd\H{o}s-Mordell inequality is implied by the following stronger DDVV-type inequality
 $$\sum^3_{r,s=1}\|[X_r,X_s]\|^2\leq \frac{1}{4}\(\sum^3_{r=1}\|X_r\|^2\)^2, \quad X_1,X_2,X_3\in V\subset o(3).$$
 Motivated by this phenomenon, we are interested in such stronger DDVV-type inequalities when the matrices are restricted to some subspaces of the regarded type.
 In Section \ref{sec4} of this paper, we establish such stronger DDVV-type inequalities when the matrices $B_1, \cdots, B_M$ are in the subspace $V_{Cs}$ (resp. $V_{Ca}$) of the space $SM(2l,\mathbb{R})$ of real symmetric $(2l\times2l)$ matrices (resp. of the space $o(l)$ of real skew-symmetric $(l\times l)$ matrices) spanned by a Clifford system $(P_0,\cdots,P_m)$ on $\mathbb{R}^{2l}$ (resp. by a Clifford algebra $(E_1,\cdots,E_{m-1})$ on $\mathbb{R}^l$). Now the optimal constant $c$ is $\frac2l\(1-\frac1N\)$
($N=\min\{m+1, M\}$) for the case of Clifford system, and is $\frac4l\(1-\frac1N\)$ ($N=\min\{m-1, M\}$) for the case of Clifford algebra (see Theorems \ref{TCsDDVV}, \ref{TCaDDVV} in Section \ref{sec4} for equality conditions).

To illustrate the number $c$ more explicitly, we briefly introduce the representation theory of Clifford algebra (cf. \cite{FKM81}). A Clifford system on $\mathbb{R}^{2l}$ can be represented by real symmetric orthogonal matrices $P_0,\cdots,P_m\in O(2l)$ satisfying $P_iP_j+P_jP_i=2\delta_{ij}I_{2l}$; a Clifford algebra on $\mathbb{R}^l$ can be represented by real skew-symmetric orthogonal matrices $E_1,\cdots,E_{m-1}\in O(l)$ satisfying $E_iE_j+E_jE_i=-2\delta_{ij}I_{l}$; they are one-to-one correspondent by setting
\begin{equation*}
P_0=\begin{pmatrix}
    I_l &  0 \\
     0 & -I_l
\end{pmatrix}, \quad
P_1=\begin{pmatrix}
  0 & I_l    \\
 I_l & 0
\end{pmatrix}, \quad
P_{\alpha+1}=\begin{pmatrix}
  0 & E_\alpha    \\
 -E_\alpha & 0
\end{pmatrix}, \quad \alpha=1,\cdots,m-1.
\end{equation*}
A Clifford system $(P_0,\cdots,P_m)$ on $\mathbb{R}^{2l}$ (resp. Clifford algebra $(E_1,\cdots,E_{m-1})$  on $\mathbb{R}^l$) can be decomposed into a direct sum of $k$ irreducible Clifford systems (resp. Clifford algebras) on $\mathbb{R}^{2\delta(m)}$ (resp. on $\mathbb{R}^{\delta(m)}$) with $l=k\delta(m)$ for $k,m\in\mathbb{N}$, where the irreducible dimension $\delta(m)$ satisfies $\delta(m+8)=16\delta(m)$ and can be listed in the following Table \ref{table3}. This means that under a uniform congruence $(Q^tP_0Q,\cdots,Q^tP_mQ)$ for some $Q\in O(2l)$, each $P_\alpha$ can be deformed into $Q^tP_\alpha Q=\diag\(P_{\alpha 1},\cdots, P_{\alpha k}\)$ with $(P_{0i},\cdots, P_{mi})$ $(i=1,\cdots,k)$ being irreducible Clifford systems on $\mathbb{R}^{2\delta(m)}$ of the same irreducible dimension. For instance, given a Clifford system $$P:= \(\begin{pmatrix}P_0^A & \\ & P_0^B\end{pmatrix},\begin{pmatrix}P_1^A
 & \\ & P_1^B\end{pmatrix},\begin{pmatrix}P_2^A & \\ &
 P_2^B\end{pmatrix}\in O(2\cdot 6)\)$$ on $\mathbb{R}^{2\cdot 6}$ with $m=2$, where $\(P_0^A,P_1^A,P_2^A\in O(2\cdot 2)\)$, $ \(P_0^B,P_1^B,P_2^B,P_3^B\in O(2\cdot 4)\)$ are Clifford systems on $\mathbb{R}^{2\cdot 2}$ with $m=2$ and on $\mathbb{R}^{2\cdot4}$ with $m=3$ respectively, there exists a uniform congruence to decompose $P$ into $k=3$ direct summands of irreducible Clifford systems on $\mathbb{R}^{2\cdot 2}$. Hence we need only consider such irreducible decompositions of Clifford systems or algebras with $l=k\delta(m)$, since the DDVV-type inequalities are invariant under uniform congruences. 
\begin{table}[h!]
\caption{Dimension $\delta(m)$ of irreducible representation of Clifford algebra}\label{table3}
\centering
\begin{tabular}{|c|c|c|c|c|c|c|c|c|c|}
\hline
$m$ & $1$ & $2$ & $3$ & $4$ & $5$ & $6$ & $7$ & $8$ & $\cdots~ m+8$\\
\hline
$\delta(m)$& $1$ &$2$ &$4$ &$4$ & $8$ & $8$ & $8$ & $8$ & $\cdots~ 16\delta(m)$\\
\hline
\end{tabular}
\end{table}

When the number $M$ of matrices is smaller than $m+1$ (resp. $m-1$), one can regard the matrices as lying in a subspace spanned by the Clifford system $(P_0,\cdots,P_{M-1})$ (resp. Clifford algebra $(E_1,\cdots,E_{M})$). Hence, only the optimal constant $c$ in the case when $M\geq m+1=N$ (resp. $M\geq m-1=N$) is of essential meaning. We summarize for this case the number $c$ with respect to the two natural numbers $k,m\in\mathbb{N}$ with $l=k\delta(m)$ in the following Table \ref{table4} for Clifford system and Table \ref{table5} for Clifford algebra.
\begin{table}[h!]
\caption{The optimal constant $c$ of DDVV-type inequalities for Clifford system}\label{table4}
\centering
\begin{tabular}{|c|c|c|c|c|c|c|c|c|c|}
\hline
$k~\backslash~ c~/~m$ & $1$ & $2$ & $3$ & $4$ & $5$ & $6$ & $7$ & $8$ & $\cdots~ m+8$\\
\hline
$1$& $1$ &$\frac{2}{3}$ &$\frac{3}{8}$ &$\frac{2}{5}$ & $\frac{5}{24}$ & $\frac{3}{14}$ & $\frac{7}{32}$ & $\frac{2}{9}$ & $\cdots~ \frac{m}{8(m+1)\delta(m)}$\\
\hline
$2$& $\frac{1}{2}$ &$\frac{1}{3}$ &$\frac{3}{16}$ &$\frac{1}{5}$ & $\frac{5}{48}$ & $\frac{3}{28}$ & $\frac{7}{64}$ & $\frac{1}{9}$ & $\cdots~ \frac{m}{16(m+1)\delta(m)}$\\
\hline
$\vdots$& $\vdots$ &$\vdots$ &$\vdots$ &$\vdots$ & $\vdots$ & $\vdots$& $\vdots$ & $\vdots$ & $\vdots$\\
\hline
$k$& $\frac{1}{k}$ &$\frac{2}{3k}$ &$\frac{3}{8k}$ &$\frac{2}{5k}$ & $\frac{5}{24k}$ & $\frac{3}{14k}$ & $\frac{7}{32k}$ & $\frac{2}{9k}$ & $\cdots~ \frac{m}{8k(m+1)\delta(m)}$\\
\hline
\end{tabular}
\end{table}
\begin{table}[h!]
\caption{The optimal constant $c$ of DDVV-type inequalities for Clifford algebra}\label{table5}
\centering
\begin{tabular}{|c|c|c|c|c|c|c|c|c|c|}
\hline
$k~\backslash~ c~/~m$ & $1$ & $2$ & $3$ & $4$ & $5$ & $6$ & $7$ & $8$ & $\cdots~ m+8$\\
\hline
$1$& -- &$0$ &$\frac{1}{2}$ &$\frac{2}{3}$  & $\frac{3}{8}$  & $\frac{2}{5}$ & $\frac{5}{12}$ & $\frac{3}{7}$ & $\cdots~ \frac{m-2}{4(m-1)\delta(m)}$\\
\hline
$2$& -- &$0$ &$\frac{1}{4}$ &$\frac{1}{3}$ & $\frac{3}{16}$ & $\frac{1}{5}$ & $\frac{5}{24}$ & $\frac{3}{14}$ & $\cdots~ \frac{m-2}{8(m-1)\delta(m)}$\\
\hline
$\vdots$& $\vdots$ &$\vdots$ &$\vdots$ &$\vdots$ & $\vdots$ & $\vdots$& $\vdots$ & $\vdots$ & $\vdots$\\
\hline
$k$& -- &$0$ &$\frac{1}{2k}$ &$\frac{2}{3k}$  & $\frac{3}{8k}$  & $\frac{2}{5k}$ & $\frac{5}{12k}$ & $\frac{3}{7k}$ & $\cdots~ \frac{m-2}{4k(m-1)\delta(m)}$\\
\hline
\end{tabular}
\end{table}

It seems that the lists of the optimal constant $c$ would possibly have some links with random matrix theory or quantum physics. However, it is just our naive and wild guess since we know nothing about that. To conclude this section, we would like to mention more about possible future studies on DDVV-type inequalities:
\begin{itemize}
\item[(1)] What is the expectation of the commutators of random matrices in certain categories like GOE, GUE, and GSE?
\item[(2)] Find more DDVV-type inequalities for matrices, Lie algebras or operators lying in certain subspaces of special interest like spaces of austere matrices (see for a special example in \cite{GTY16}).
\item[(3)] Whether the nonnegative polynomial defined by the DDVV-type inequalities (by $F(B_1,\cdots,B_m):=RHS-LHS$) is a sum of squares of quadratic forms on the matrices in the regarded types? This would provide more examples on Hilbert's 17th problem (see other examples constructed by Clifford systems in \cite{GT18}).
\end{itemize}

\section{DDVV-type inequality for complex matrices}\label{sec2}
We have already known the DDVV-type inequality for (skew-)Hermitian matrices
and its equality condition (cf. \cite{GXYZ17}). In this section, we firstly give a simpler proof of this result to illustrate our main technique of this paper.
\begin{thm}\label{Thm-Hermitian}\cite[Theorem 1.1]{GXYZ17}
Let $B_1,\cdots,B_m$ be $n\times n$ Hermitian matrices $(n\geq2,~ m\geq3)$.
Then
\begin{equation}\label{HDDVV}
\sum^m_{r,s=1}\|\[B_r,B_s\]\|^2\leq \frac43\(\sum^m_{r=1}\|B_r\|^2\)^2.
\end{equation}
\end{thm}
\begin{proof}
The main idea is to realize the complex matrices by real matrices and use the known results of the DDVV-type inequalities.
Let
$$\begin{aligned}
\Phi: M(n,\mathbb{C})&\longrightarrow M(2n,\mathbb{R}),\\
X=X_1+X_2\mathbf{i}&\longmapsto
\begin{pmatrix}
X_2 & X_1 \\
-X_1 & X_2 \\
\end{pmatrix},
\end{aligned}$$
where $X_1, X_2 \in M(n,\mathbb{R})$.
For $X, Y \in M(n,\mathbb{C})$, One can verify directly that
\begin{equation}\label{HDDVV2}
\|\Phi(X)\|^2=2\|X\|^2,
\end{equation}
\begin{equation}\label{HDDVV3}
\|\Phi(-\mathbf{i}X)\|^2=\|\Phi(X)\|^2,
\end{equation}
\begin{equation}\label{HDDVV4}
\[\Phi(X),\Phi(Y)\]=\Phi(-\mathbf{i}\[X,Y\]).
\end{equation}

As $B_r$ is Hermitian means that for every $1\leq r \leq m$,
all the $\Phi(B_r)$ are skew-symmetric.
Therefore we can use the known DDVV-type inequality for real skew-symmetric matrices (for $2n\geq4$ and $m\geq3$):
\begin{equation}\label{HDDVV5}
\sum^m_{r,s=1}\|\[\Phi(B_r),\Phi(B_s)\]\|^2
\leq \frac23\(\sum^m_{r=1}\|\Phi(B_r)\|^2\)^2.
\end{equation}
It follows from (\ref{HDDVV2}), (\ref{HDDVV3}), (\ref{HDDVV4}) and (\ref{HDDVV5}) that
$$\begin{aligned}
\sum^m_{r,s=1}\|\[B_r,B_s\]\|^2 &=\frac12\sum^m_{r,s=1}\|\Phi(\[B_r,B_s\])\|^2\\
&=\frac12\sum^m_{r,s=1}\|\[\Phi(B_r),\Phi(B_s)\]\|^2\\
&\leq \frac13\(\sum^m_{r=1}\|\Phi(B_r)\|^2\)^2=\frac43\(\sum^m_{r=1}\|B_r\|^2\)^2.
\end{aligned}$$
The proof is complete.
\end{proof}

\begin{rem}\label{rem-HDDVV}
When $B_1,\cdots,B_m$ consist of Hermitian matrices and skew-Hermitian matrices,
$(\ref{HDDVV})$ still holds. This is because that the norms are invariant under the multiplication by $\mathbf{i}$,
and a skew-Hermitian matrix multiplying $\mathbf{i}$ is a Hermitian matrix.
\end{rem}

To describe the equality condition in the next theorem, we put $K(n, m):=U(n)\times O(m)$.
A $K(n, m)$ action on a family of matrices $(A_1,\cdots,A_m)$ is given by
$$(P,R)\cdot(A_1,\cdots,A_m):=\left(\sum_{j=1}^mR_{j1}P^*A_jP,\cdots,\sum_{j=1}^mR_{jm}P^*A_jP\right),$$
for $(P,R)\in K(n, m)$,
where $R=(R_{jk})\in O(m)$ acts as a rotation on the matrix tuple $(P^*A_1P,\cdots,P^*A_mP)$, $P^*=\overline{P}^t$ is the conjugate transpose.
Using the technique in the previous proof, we obtain the DDVV-type inequality for general complex matrices.
\begin{thm}\label{TcDDVV}
Let $B_1,\cdots,B_m$ be arbitrary $n\times n$ complex matrices $(n\geq2)$.
\begin{enumerate}
\item If $m\geq3$, then
\begin{equation}\label{cDDVV}
\sum^m_{r,s=1}\|\[B_r,B_s\]\|^2\leq \frac43\(\sum^m_{r=1}\|B_r\|^2\)^2.
\end{equation}
For $1\leq r \leq m$, let $B_r^1=\frac12(B_r+B_r^*), B_r^2=\frac12(B_r-B_r^*)$.
The equality holds if and only if $\sum^m_{r=1}\[B_r,B_r^*\]=0$ and
there exists a $(P,R)\in K(n, 2m)$ such that
$$(P,R)\cdot(B_1^1, \cdots, B_m^1, \mathbf{i}B_1^2, \cdots, \mathbf{i}B_m^2)=
(\diag(H_1,0), \diag(H_2,0), \diag(H_3,0), 0, \cdots, 0),$$ where for some $\lambda\geq0$,
$$H_1:=
\begin{pmatrix}
\lambda & 0 \\
0 & -\lambda \\
\end{pmatrix}, \quad
H_2:=
\begin{pmatrix}
0 & \lambda \\
\lambda & 0 \\
\end{pmatrix},\quad
H_3:=
\begin{pmatrix}
0 & -\lambda\mathbf{i} \\
\lambda\mathbf{i} & 0 \\
\end{pmatrix}.
$$

\item If $m=2$, then
$$\sum^2_{r,s=1}\|\[B_r,B_s\]\|^2\leq \(\sum^2_{r=1}\|B_r\|^2\)^2.$$
The equality holds if and only if there exists a unitary matrix $U$
such that $B_1=U^*\diag\(\widetilde{B_1}, 0\)U, B_2=U^*\diag\(\widetilde{B_2}, 0\)U$,
where $\widetilde{B_1}, \widetilde{B_2}\in M(2, \mathbb{C})$
with $\left\|\widetilde{B_1}\right\|=\left\|\widetilde{B_2}\right\|,
\left\<\widetilde{B_1}, \widetilde{B_2}\right\>=0, \tr\(\widetilde{B_1}\)=\tr\(\widetilde{B_2}\)=0$.
\end{enumerate}
\end{thm}
\begin{proof}
The case (2) for $m=2$ is implied by the BW inequality for complex matrices (cf. \cite{BW08}) and its equality condition (cf. \cite{CVW10}). We prove the case (1) as follows.

For $1\leq r \leq m$, let $B_r=B_r^1+B_r^2$,
where $B_r^1$ is a Hermitian matrix, $B_r^2$ is a skew-Hermitian matrix.
Because Hermitian matrices are orthogonal to skew-Hermitian matrices, we have
\begin{equation}\label{cDDVV2}
\|B_r\|^2=\|B_r^1\|^2+\|B_r^2\|^2.
\end{equation}
Using this decomposition of complex matrices, we calculate the commutators directly:
$$\begin{aligned}
\|\[B_r,B_s\]\|^2 
&=\sum^2_{i,j=1}\|\[B_r^i,B_s^j\]\|^2-2\left<\[B_r^1,B_r^2\],\[B_s^1,B_s^2\]\right>,
\end{aligned}$$
where we have also used the known Jacobi identity. 
Note that commutators of two symmetry or two skew-symmetry matrices are skew-symmetry,
and a commutator of a symmetry matrix with a skew-symmetry matrix is symmetry.
It follows that
\begin{equation}\label{cDDVV3}
\sum^m_{r,s=1}\|\[B_r,B_s\]\|^2
=\sum^m_{r,s=1}\sum^2_{i,j=1}\|\[B_r^i,B_s^j\]\|^2
-2\left\|\sum^m_{r=1}\[B_r^1,B_r^2\]\right\|^2.
\end{equation}
Since $B_1^1,\cdots, B_m^1, B_1^2,\cdots, B_m^2$
are Hermitian or skew-Hermitian matrices,
by Theorem \ref{Thm-Hermitian} and Remark \ref{rem-HDDVV} we have the DDVV-type inequality:
\begin{equation}\label{cDDVV4}
\sum^m_{r,s=1}\sum^2_{i,j=1}\|\[B_r^i,B_s^j\]\|^2 \leq
\frac43\(\sum^m_{r=1}\sum^2_{i=1}\|B_r^i\|^2\)^2
\end{equation}
Then by (\ref{cDDVV2}), (\ref{cDDVV3}) and (\ref{cDDVV4}), we have
$$\begin{aligned}
\sum^m_{r,s=1}\|\[B_r,B_s\]\|^2
&= \sum^m_{r,s=1}\sum^2_{i,j=1}\|\[B_r^i,B_s^j\]\|^2-2\left\|\sum^m_{r=1}\[B_r^1,B_r^2\]\right\|^2\\
&\leq \frac43\(\sum^m_{r=1}\sum^2_{i=1}\|B_r^i\|^2\)^2-2\left\|\sum^m_{r=1}\[B_r^1,B_r^2\]\right\|^2\\
&=\frac43\(\sum^m_{r=1}\|B_r\|^2\)^2-\frac12\left\|\sum^m_{r=1}\[B_r,B_r^*\]\right\|^2.
\end{aligned}$$
The last inequality implies (\ref{cDDVV}) and the first equality condition. As for the second equality condition, we just apply the equality condition of the DDVV-type inequality (\ref{cDDVV4}) for Hermitian matrices $B_1^1, \cdots, B_m^1, \mathbf{i}B_1^2, \cdots, \mathbf{i}B_m^2$ given by \cite{GXYZ17}.

The proof is complete.
\end{proof}

\begin{rem}
Let
$$B_1:=
\begin{pmatrix}
1 & 0 \\
0 & -1 \\
\end{pmatrix}, \quad
B_2:=
\begin{pmatrix}
0 & 1 \\
1 & 0 \\
\end{pmatrix}, \quad
B_3:=
\begin{pmatrix}
0 & -1 \\
1 & 0 \\
\end{pmatrix},$$
then$$\sum^3_{r,s=1}\|\[B_r,B_s\]\|^2\ = \frac43\(\sum^3_{r=1}\|B_r\|^2\)^2, \quad \sum^2_{r,s=1}\|\[B_r,B_s\]\|^2\ = \(\sum^2_{r=1}\|B_r\|^2\)^2.$$
Hence the optimal constants for the real matrices case and the complex matrices case are both
$\frac43$ for $m\geq3$, and $1$ for $m=2$.
\end{rem}

When $m=3$, we have even the following simpler proof of the DDVV-type inequality (\ref{cDDVV}) by using the BW inequality.
For $1\leq r, s \leq 3$, by the BW inequality and the inequality of arithmetic and geometric means, we have
$$\|\[B_r,B_s\]\|^2\leq 2\|B_r\|^2\|B_s\|^2\leq \frac43\|B_r\|^2\|B_s\|^2+\frac13\(\|B_r\|^4+\|B_s\|^4\),$$
which directly shows
$$\sum^3_{r,s=1}\|\[B_r,B_s\]\|^2\leq
\frac43\sum^3_{r,s=1, r\neq s}\|B_r\|^2\|B_s\|^2+\frac13\sum^3_{r,s=1, r\neq s}\(\|B_r\|^4+\|B_s\|^4\)
=\frac43\(\sum^3_{r=1}\|B_r\|^2\)^2. $$

When $B_1,\cdots,B_m$ are complex symmetric or complex skew-symmetric matrices,
we can get smaller optimal constants by slightly changing the proof of (\ref{cDDVV}).
\begin{cor}\label{cor-sym}
Let $B_1,\cdots,B_m$ be $n\times n$ complex symmetric matrices $(n\geq2)$, then
\begin{equation}\label{csymDDVV}
\sum^m_{r,s=1}\|\[B_r,B_s\]\|^2\leq \(\sum^m_{r=1}\|B_r\|^2\)^2.
\end{equation}
For $1\leq r \leq m$, let $A_r^1=\frac12(B_r+\overline{B_r}), A_r^2=\frac1{2\mathbf{i}}(B_r-\overline{B_r})$.
The equality holds if and only if $\sum^m_{r=1}\[B_r,\overline{B_r}\]=0$ and
there exists a $(P,R)\in K(n, 2m)$ (where $P\in O(n)$) such that
$$(P,R)\cdot(A_1^1, \cdots, A_m^1, A_1^2, \cdots, A_m^2)=
(\diag(H_1,0), \diag(H_2,0), 0, \cdots, 0),$$ where for some $\lambda\geq0$,
$$H_1:=
\begin{pmatrix}
\lambda & 0 \\
0 & -\lambda \\
\end{pmatrix}, \quad
H_2:=
\begin{pmatrix}
0 & \lambda \\
\lambda & 0 \\
\end{pmatrix}.
$$
\end{cor}
The proof is essentially the same as for Theorem \ref{TcDDVV} but based on the decomposition into the sum of a real symmetric matrix and a purely imaginary symmetric matrix.
As the imaginary unit does not change norms, we can utilize the DDVV inequality for real symmetric matrices.

In the same way, we can obtain the next corollary by the DDVV-type inequality
for real skew-symmetric matrices. The proof is omitted here.
\begin{cor}\label{cor-skew}
Let $B_1,\cdots,B_m$ be $n\times n$ complex skew-symmetric matrices, $m\geq 3$.
\begin{enumerate}
\item If $n=3$, then
\begin{equation*}\label{cskewDDVV1}
\sum^m_{r,s=1}\|\[B_r,B_s\]\|^2\leq \frac13\(\sum^m_{r=1}\|B_r\|^2\)^2.
\end{equation*}
For $1\leq r \leq m$, let $A_r^1=\frac12(B_r+\overline{B_r}), A_r^2=\frac1{2\mathbf{i}}(B_r-\overline{B_r})$.
The equality holds if and only if $\sum^m_{r=1}\[B_r,\overline{B_r}\]=0$ and
there exists a $(P,R)\in K(n, 2m)$  (where $P\in O(n)$) such that
$$(P,R)\cdot(A_1^1, \cdots, A_m^1, A_1^2, \cdots, A_m^2)=
(\diag(C_1,0), \diag(C_2,0), \diag(C_3,0), 0, \cdots, 0),$$ where for some $\lambda\geq0$,
$$C_1:=
\begin{pmatrix}
0 & \lambda & 0 \\
-\lambda & 0 & 0 \\
0 & 0 & 0 \\
\end{pmatrix}, \quad
C_2:=
\begin{pmatrix}
0 & 0 & \lambda \\
0 & 0 & 0 \\
-\lambda & 0 & 0 \\
\end{pmatrix},\quad
C_3:=
\begin{pmatrix}
0 & 0 & 0 \\
0 & 0 & \lambda \\
0 & -\lambda & 0 \\
\end{pmatrix}.
$$

\item If $n\geq 4$, then
\begin{equation*}\label{cskewDDVV2}
\sum^m_{r,s=1}\|\[B_r,B_s\]\|^2\leq \frac23\(\sum^m_{r=1}\|B_r\|^2\)^2.
\end{equation*}
For $1\leq r \leq m$, let $A_r^1=\frac12(B_r+\overline{B_r}), A_r^2=\frac1{2\mathbf{i}}(B_r-\overline{B_r})$.
The equality holds if and only if $\sum^m_{r=1}\[B_r,\overline{B_r}\]=0$ and
there exists a $(P,R)\in K(n, 2m)$  (where $P\in O(n)$) such that
$$(P,R)\cdot(A_1^1, \cdots, A_m^1, A_1^2, \cdots, A_m^2)=
(\diag(D_1,0), \diag(D_2,0), \diag(D_3,0), 0, \cdots, 0),$$ where for some $\lambda\geq0$,
$$D_1:=
\begin{pmatrix}
0 & \lambda & 0 & 0 \\
-\lambda & 0 & 0 & 0 \\
0 & 0 & 0 & \lambda \\
0 & 0 & -\lambda & 0 \\
\end{pmatrix}, \quad
D_2:=
\begin{pmatrix}
0 & 0 & \lambda &0 \\
0 & 0 & 0 & -\lambda \\
-\lambda & 0 & 0 & 0 \\
0 & \lambda & 0 & 0 \\
\end{pmatrix},\quad
D_3:=
\begin{pmatrix}
0 & 0 & 0 & \lambda \\
0 & 0 & \lambda & 0 \\
0 & -\lambda & 0 & 0 \\
-\lambda & 0 & 0 & 0 \\
\end{pmatrix}.
$$
\end{enumerate}
\end{cor}

\section{DDVV-type inequality for quaternionic matrices}\label{sec3}
In this section, we generalize the BW inequality and the DDVV inequality to quaternionic matrices by the same idea as in the last section.
In this case, it turns out that both of
the optimal constants $c$ are double of that for complex matrices,
mainly because the multiplication of $\mathbf{i}, \mathbf{j}, \mathbf{k}$ is anti-commutative.
The proof is carried out simply by mapping a quaternionic matrix into a complex matrix and then using the known inequalities for complex matrices
in a fashion analogous to what we have done with proving Theorem \ref{Thm-Hermitian}.

Let
$$\begin{aligned}
\Psi: M(n,\mathbb{H})&\longrightarrow M(2n,\mathbb{C}),\\
X=X_1+X_2\mathbf{j}&\longmapsto
\begin{pmatrix}
X_1 & X_2 \\
-\overline{X_2} & \overline{X_1} \\
\end{pmatrix},
\end{aligned}$$
where  $X_1, X_2 \in M(n,\mathbb{C})$.
It is easy to see
\begin{equation}\label{qDDVV2}
\|\Psi(X)\|^2=2\|X\|^2.
\end{equation}
For $X, Y \in M(n,\mathbb{H})$, let
$$A_1:=[X_1, Y_1]-X_2\overline{Y_2}+Y_2\overline{X_2}, \quad
A_2:=X_1Y_2-Y_2\overline{X_1}+X_2\overline{Y_1}-Y_1X_2.$$
Then direct calculations show
 $$[X,Y]=A_1+A_2\mathbf{j},\quad
\[\Psi(X),\Psi(Y)\]=
\begin{pmatrix}
A_1 & A_2 \\
-\overline{A_2} & \overline{A_1} \\
\end{pmatrix}. $$
Remember $\mathbf{j}\mathbf{i}=-\mathbf{i}\mathbf{j}$ when moving the $\mathbf{j}$ to the right.
Therefore $\Psi$ preserves the commutator:
\begin{equation}\label{qDDVV3}
\[\Psi(X),\Psi(Y)\]=\Psi(\[X,Y\]).
\end{equation}

\begin{thm}\label{TqBW}
Let $X,Y$ be arbitrary $n\times n$ quaternionic matrices, then
\begin{equation}\label{qBW}
\|\[X,Y\]\|^2\leq 4\|X\|^2\|Y\|^2.
\end{equation}
The equality holds if and only if either $\|X\|\|Y\|=0$ or there exists a unitary matrix $U$
such that $\Psi(X)=U^*\diag(X_0, 0)U, \Psi(Y)=U^*\diag(Y_0, 0)U$,
where $X_0, Y_0\in M(2, \mathbb{C})$ with $\<X_0,Y_0\>=0, \tr(X_0)=\tr(Y_0)=0$, and $\rank X_0=\rank Y_0=2$.
\end{thm}
\begin{proof}
Since $\Psi(X), \Psi(Y)$ are complex matrices , we have the BW inequality:
\begin{equation}\label{qDDVV4}
\|\[\Psi(X),\Psi(Y)\]\|^2\leq 2\|\Psi(X)\|^2\|\Psi(Y)\|^2.
\end{equation}
Combining (\ref{qDDVV2}), (\ref{qDDVV3}) and (\ref{qDDVV4}), we obtain immediately
$$
\|\[X,Y\]\|^2=\frac12\|\Psi(\[X,Y\])\|^2
=\frac12\|\[\Psi(X),\Psi(Y)\]\|^2
\leq \|\Psi(X)\|^2\|\Psi(Y)\|^2
=4\|X\|^2\|Y\|^2.
$$

By the equality condition of
the BW inequality for complex matrices (cf. \cite{CVW10}), the equality holds if and only if either $\|X\|\|Y\|=0$ or there exists a unitary matrix $U$
such that $\Psi(X)=U^*\diag(X_0, 0)U$, $\Psi(Y)=U^*\diag(Y_0, 0)U$,
where $X_0, Y_0\in M(2, \mathbb{C})$ with $\<X_0,Y_0\>=0$, $\tr(X_0)=\tr(Y_0)=0$. In the second case this implies that $\rank\Psi(X)\leq2$ and $\rank\Psi(Y)\leq2$, whereas $\rank\Psi(X)=2\rank X$ and $\rank\Psi(Y)=2\rank Y$ (see \cite{Z97} for the definition of the rank of quaternionic matrices and for a survey of various related results). Hence $\rank X_0=\rank\Psi(X)=2\rank X=2$ and $\rank Y_0=\rank\Psi(Y)=2\rank Y=2$. The proof is complete.
\end{proof}

\begin{rem}
Let $X=\mathbf{i},~ Y=\mathbf{j}$, then
$$\|\[X,Y\]\|^2= 4\|X\|^2\|Y\|^2. $$
Hence $4$ is the optimal constant for the BW-type inequality for quaternionic matrices.
Moreover, for any $x\in\mathbb{R}^n$ and $\lambda\in\mathbb{R}$, let $X=xx^t\mathbf{i},~Y=\lambda xx^t\mathbf{j}\in M(n,\mathbb{H})$, we have
$$\|\[X,Y\]\|^2= 4\|X\|^2\|Y\|^2. $$
\end{rem}
The matrices in the maximal pair $(X,Y)$ in the remark above indeed have rank one, which is necessary as shown in the proof of Theorem \ref{TqBW}. However, the rank condition along one is not sufficient.
For example, let $X=\begin{pmatrix}
0 &  \mathbf{i}\\
0 & 0 \\
\end{pmatrix}$, or $\begin{pmatrix}
\mathbf{i} &  \mathbf{i}\\
0 & 0 \\
\end{pmatrix}$, then for any $Y\in M(2,\mathbb{H})$, $$\|\[X,Y\]\|^2 < 4\|X\|^2\|Y\|^2.$$

 Recall that in the real or complex case, $4$ is trivially an upper bound (though not optimal) by using the triangle inequality and sub-multiplicativity.
 In fact, this is also applicable in the quaternionic case as follows:
 \begin{equation}\label{qBWpf}
 \begin{aligned}
\|\[X,Y\]\|^2&=\|XY-YX\|^2=\|XY\|^2+\|YX\|^2-2\langle XY, YX\rangle \\
&\leq \(\|XY\|+\|YX\|\)^2\leq 2\(\|XY\|^2+\|YX\|^2\)\leq 4\|X\|^2\|Y\|^2.
\end{aligned}
\end{equation}
Here the last inequality follows from the sub-multiplicativity for quaternionic matrices:
\begin{lem}\label{submult}
Let $X,Y\in M(n,\mathbb{H})$, then $\|XY\|\leq \|X\| \|Y\|$. The equality holds if and only if
either $\|X\| \|Y\|=0$, or $X=au^*$ and $Y=ub^*$ for some column vectors $a,b\in\mathbb{H}^n$ and a unit column vector $u\in\mathbb{H}^n$.
\end{lem}
\begin{proof}
Since $X^*X$ is quaternionic Hermitian and positive semi-definite, there exists a quaternionic unitary matrix $U\in Sp(n)$ such that $U^*X^*XU=\diag (\lambda_1,\cdots, \lambda_n)$ with nonnegative real numbers $\lambda_1\geq\cdots\geq \lambda_n\geq0$ (cf. \cite{Z97}). Then
$$\begin{aligned}
\|XY\|^2&=\langle XY, XY\rangle=\langle X^*X, YY^*\rangle\\
&=\langle U^*X^*XU, U^*YY^*U\rangle=\sum_{i=1}^n\lambda_i \(\sum_{k=1}^n\|\sum_{j=1}^n\overline{u_{ji}}y_{jk}\|^2\)\\
&\leq \(\sum_{i=1}^n\lambda_i\) \(\sum_{k,l=1}^n\|\sum_{j=1}^n\overline{u_{jl}}y_{jk}\|^2\)=\|X\|^2\|Y\|^2,
\end{aligned}$$
where $u_{ij}, y_{ij}$ are entries of $U$ and $Y$. Here we have used the fact:
$$\langle XY, XY\rangle=\Re \tr(XYY^*X^*)=\Re \tr(X^*XYY^*)=\langle X^*X, YY^*\rangle,$$
although in general $\tr(AB)\neq \tr(BA)$ for quaternionic matrices. It follows from the last inequality that the equality holds if and only if
either $\|X\| \|Y\|=0$, or $\lambda_1>0=\lambda_2=\cdots=\lambda_n$ and $U^*Y=(b,0,\cdots,0)^*$ for some column vector $b\in\mathbb{H}^n$. Hence the equality holds if and only if
either $\|X\| \|Y\|=0$, or $X=(a,0,\cdots,0)U^*$ and $Y=U(b,0,\cdots,0)^*$ for some quaternionic unitary matrix $U\in Sp(n)$ and some column vectors $a,b\in\mathbb{H}^n$.
Let $u\in\mathbb{H}^n$ be the first column of $U$. The proof is complete.
\end{proof}
Meanwhile, we have the Cauchy-Schwarz inequality for quaternions:
\begin{lem}\label{qCS}
Let $a,b\in\mathbb{H}^n$ be quaternionic vectors. Then $\|\langle a, b\rangle_\mathbb{H}\|\leq \|a\|\|b\|$, where $\langle a, b\rangle_\mathbb{H}:=b^*a$ is the quaternionic Hermitian product. The equality holds if and only if either $\|a\|\|b\|=0$ or $b=a\sigma$ for some quaternion $\sigma\in\mathbb{H}$.
\end{lem}
\begin{proof}
The proof is similar to that for real or complex cases (cf. \cite{M70}, where the equality condition of Theorem 6 in page 43 is incorrect, namely, the real $\lambda$ there should be a quaternion on the right). For the sake of completeness, we give a proof following \cite{CMW13}.

Without loss of generality, we assume $\|a\|=\|b\|=1$. Let $\sigma:=a^*b=\langle b, a\rangle_\mathbb{H}$. Then for any $t\in\mathbb{R}$, we have
$$0\leq \|tb-a\sigma\|^2=t^2-2t\|\sigma\|^2+\|\sigma\|^2,$$
which implies $$\|\langle a, b\rangle_\mathbb{H}\|^2=\|\sigma\|^2\leq1= \|a\|^2\|b\|^2.$$
The equality holds if and only if $\|b-a\sigma\|^2=0$, i.e., $b=a\sigma$.
\end{proof}

Based on the proof (\ref{qBWpf}) and Lemmas \ref{submult} and \ref{qCS}, we are able to characterize the maximal pairs of the BW-type inequality (\ref{qBW}) in quaternion domain (pointed out to us by D. Wenzel), thus strengthening Theorem \ref{TqBW}.
\begin{thm}\label{TqBW2}
Let $X,Y\in M(n,\mathbb{H})$, then $\|[X,Y]\|^2\leq 4\|X\|^2 \|Y\|^2$. The equality holds if and only if
either $\|X\| \|Y\|=0$, or $X=upu^*$ and $Y=uqu^*$ for some unit column vector $u\in\mathbb{H}^n$, where $p,q\in\mathbb{H}$ are 
purely imaginary quaternions that have real-orthogonal vector representations in the canonical basis.
\end{thm}

\begin{rem}
The maximal pair $(X,Y)$ can be rewritten as $X=U\diag(p,0,\cdots,0)U^*$, $Y=U\diag(q,0,\cdots,0)U^*$ for some quaternionic unitary matrix $U\in Sp(n)$. The condition on $p,q$ is equivalent to the anti-commutativity $pq=-qp$, which cannot happen in the real or complex cases.
\end{rem}
\begin{proof}
Without loss of generality, we assume $\|X\|=\|Y\|=1$. It follows from the last inequality of (\ref{qBWpf}) that $\|XY\|=1$, and thus from Lemma \ref{submult} that $X=au^*$ and $Y=ub^*$ for some unit column vectors $a,b,u\in\mathbb{H}^n$. Again by the last inequality of (\ref{qBWpf}), $\|YX\|^2=\|b^*a\|^2=1$, which by Lemma \ref{qCS} implies that $b=a\sigma$ for some unit quaternion $\sigma\in\mathbb{H}$. Then by the first inequality of (\ref{qBWpf}), we have $$u\sigma^*u^*=YX=-XY=-a\sigma^*a^*.$$
Thus $\sigma=-(u^*a)\sigma (u^*a)^*$, which implies that $\|u^*a\|=1$ and thus by Lemma \ref{qCS} again we have $a=up$ for some unit quaternion $p\in\mathbb{H}$.
Then $\sigma=-p\sigma p^*$, i.e., $p\sigma=-\sigma p$. It is easily seen that $\Re p=\Re \sigma=0$. Let $q=\sigma^*p^*\in\mathbb{H}$. Then $\Re q=0$, $pq=-qp$, $X=upu^*$ and $Y=uqu^*$.
The proof is complete.
\end{proof}

Now we come to prove the DDVV-type inequality for quaternionic matrices.
\begin{thm}\label{TqDDVV}
Let $B_1,\cdots,B_m$ be arbitrary $n\times n$ quaternionic matrices.
\begin{enumerate}
\item If $m\geq3$, then
\begin{equation*}\label{qDDVV}
\sum^m_{r,s=1}\|\[B_r,B_s\]\|^2\leq \frac83\(\sum^m_{r=1}\|B_r\|^2\)^2.
\end{equation*}
For $1\leq r \leq m$, let $B_r^1=\frac12(\Psi(B_r)+\Psi(B_r)^*), B_r^2=\frac12(\Psi(B_r)-\Psi(B_r)^*)$.
The equality holds if and only if $\sum^m_{r=1}\[\Psi(B_r),\Psi(B_r)^*\]=0$
and there exists a $(P,R)\in K(n, 2m)$ such that
$$(P,R)\cdot(B_1^1, \cdots, B_m^1, \mathbf{i}B_1^2, \cdots, \mathbf{i}B_m^2)=
(\diag(H_1,0), \diag(H_2,0), \diag(H_3,0), 0, \cdots, 0),$$ where for some $\lambda\geq0$,
$$H_1:=
\begin{pmatrix}
\lambda & 0 \\
0 & -\lambda \\
\end{pmatrix}, \quad
H_2:=
\begin{pmatrix}
0 & \lambda \\
\lambda & 0 \\
\end{pmatrix}, \quad
H_3:=
\begin{pmatrix}
0 & -\lambda\mathbf{i} \\
\lambda\mathbf{i} & 0 \\
\end{pmatrix}.
$$

\item If $m=2$, then
$$\sum^2_{r,s=1}\|\[B_r,B_s\]\|^2\leq 2\(\sum^2_{r=1}\|B_r\|^2\)^2. $$
The equality holds if and only if  $B_1=upu^*$ and $B_2=uqu^*$ for some unit column vector $u\in\mathbb{H}^n$, where $p,q\in\mathbb{H}$ are orthogonal imaginary quaternions and $\|p\|=\|q\|$.
\end{enumerate}
\end{thm}
\begin{proof}
The case (2) is implied by Theorem \ref{TqBW2}. Case (1) follows the same way as Theorem \ref{Thm-Hermitian}, but uses the fitting DDVV-type inequality with a different constant.

\end{proof}

\begin{rem}
Let $B_1=\mathbf{i}, B_2=\mathbf{j}, B_3=\mathbf{k}$, then
$$\sum^3_{r,s=1}\|\[B_r,B_s\]\|^2\ = \frac83\(\sum^3_{r=1}\|B_r\|^2\)^2, $$
$$\sum^2_{r,s=1}\|\[B_r,B_s\]\|^2\ = 2\(\sum^2_{r=1}\|B_r\|^2\)^2. $$
Hence the optimal constants for the quaternionic matrices case
and the quaternionic skew-Hermitian matrices case are both
$\frac83$ for $m\geq3$, and $2$ for $m=2$. The equality condition could be also written in quaternion domain as in Theorem \ref{TqBW2}. A maximal triple $(B_1,B_2,B_3)$ determines the $(P, R)$-action, and all others must be zero. The surviving matrices should be in the form
$$B_r=uq_ru^*\in M(n,\mathbb{H}), \quad r=1,2,3,$$
for some unit column vector $u\in\mathbb{H}^n$ and $q_1,q_2,q_3\in\mathbb{H}$ are orthogonal imaginary quaternions with the same norm.
Also note that one can try to tackle the transition from the real to complex matrices in a similar, natural way. However, the doubled constant turns out not to be sharp in this case.
\end{rem}

\section{DDVV-type inequality for Clifford system and Clifford algebra}\label{sec4}
As introduced in Section \ref{sec-int}, inspired by the Erd\H{o}s-Mordell inequality, in this section we establish the DDVV-type inequalities for matrices in the subspaces spanned by a Clifford system or a Clifford algebra.
The following lemma will be used in the proof of both Theorem \ref{TCsDDVV} and Theorem \ref{TCaDDVV}.
\begin{lem}\label{Clifford}
Let $I, R$ be finite index sets.
For any $i\in I$ and any $r\in R$, $b^r_i$ is a real number.
Then
$$\sum_{r, s\in R}\(\sum_{i\in I}b^r_ib^s_i\)^2
\geq \frac1N\(\sum_{r\in R}\sum_{i\in I}(b^r_i)^2\)^2, $$
where $N=\min\{|I|, |R|\}$. 
The condition for equality has two cases.

Case $(1)$: When $|I|\leq |R|$, the equality holds if and only if
$\sum_{r\in R}b^r_ib^r_j=0$ for any $i, j\in I$ with $i\neq j$, and
$\sum_{r\in R}(b^r_i)^2$ has the same value for all $i\in I$.

Case $(2)$: When $|I|\geq |R|$, the equality holds if and only if
$\sum_{i\in I}b^r_ib^s_i=0$ for any $r, s\in R$ with $r\neq s$, and
$\sum_{i\in I}(b^r_i)^2$ has the same value for all $r\in R$.

\end{lem}
\begin{proof}
Case (1): When $|I|\leq |R|$, by the Cauchy inequality,
$$\sum_{r, s\in R}\(\sum_{i\in I}b^r_ib^s_i\)^2
=\sum_{i, j\in I}\(\sum_{r\in R}b^r_ib^r_j\)^2
\geq \sum_{i\in I}\(\sum_{r\in R}(b^r_i)^2\)^2
\geq \frac1{|I|}\(\sum_{i\in I}\sum_{r\in R}(b^r_i)^2\)^2. $$

Case(2): When $|I|\geq |R|$, also by the Cauchy inequality,
$$\sum_{r, s\in R}\(\sum_{i\in I}b^r_ib^s_i\)^2
\geq \sum_{r\in R}\(\sum_{i\in I}(b^r_i)^2\)^2
\geq \frac1{|R|}\(\sum_{r\in R}\sum_{i\in I}(b^r_i)^2\)^2. $$
The equality condition follows immediately from that of the Cauchy inequality.
\end{proof}
\begin{rem}
Let $B=(b^r_s)\in M(|R|, |I|)$. Then the inequality of the lemma is just
$$\|BB^t\|^2=\|B^tB\|^2\geq \frac{1}{N}\|B\|^4, \quad N=\min\{|I|, |R|\}.$$
The equality holds if and only if either $B^tB=\lambda I_N$ or $BB^t=\lambda I_N$ for $\lambda\geq0$.
\end{rem}


\begin{thm}\label{TCsDDVV}
Let $(P_0, P_1, \cdots, P_m)$ be a Clifford system on $\mathbb{R}^{2l}$, i.e., $P_0,\cdots,P_m\in O(2l)$ are real symmetric orthogonal matrices satisfying $P_iP_j+P_jP_i=2\delta_{ij}I_{2l}$.
Let $B_1,\cdots,B_M \in span\{P_0, P_1, \cdots, P_m\}$, then
\begin{equation}\label{CsDDVV}
\sum^M_{r,s=1}\|\[B_r,B_s\]\|^2\leq \frac2l\(1-\frac1N\)\(\sum^M_{r=1}\|B_r\|^2\)^2, \quad
N=\min\{m+1, M\}.
\end{equation}
The condition for equality has two cases.

Case $(1)$: When $m+1\leq M$, the equality holds if and only if
$p_0, \cdots, p_{m}$ are orthogonal vectors with the same norm,
where $p_i:=\(\<P_i, B_1\>, \cdots, \<P_i, B_M\>\)\in\mathbb{R}^M$.

Case $(2)$: When $m+1\geq M$, the equality holds if and only if
$B_1,\cdots,B_M$ are orthogonal matrices with the same Frobenius norm.

\end{thm}
\begin{proof}
For $1\leq r \leq M$, let $B_r=\sum^{m}_{i=0}b^r_iP_i$,
where $b^r_0, \cdots, b^r_m$ are real numbers. Since
\begin{equation}\label{CsDDVV2}
P_iP_j+P_jP_i=2\delta_{ij}I_{2l}, \quad i,j=0,\cdots,m,
\end{equation}
we have
$$\[B_r,B_s\]=\sum^{m}_{i,j=0, i\neq j}b^r_ib^s_j\[P_i,P_j\]
=2\sum^{m}_{i,j=0, i\neq j}b^r_ib^s_jP_iP_j, $$
$$\|\[B_r,B_s\]\|^2=4\left\|\sum^{m}_{i,j=0, i\neq j}b^r_ib^s_jP_iP_j\right\|^2
=4\sum^{m}_{\substack{i, j, k, l=0,\\ i\neq j, k\neq l}}b^r_ib^s_jb^r_kb^s_l\tr(P_iP_jP_lP_k). $$
By (\ref{CsDDVV2}) and $\tr(P_iP_j)=\tr(P_jP_i)$ for all $0\leq i,j\leq m$, we have
$$\begin{aligned}
\Delta_1:&=\sum^{m}_{i, j, k, l=0}b^r_ib^s_jb^r_kb^s_l\tr(P_iP_jP_lP_k)\\
&=\sum^{m}_{\substack{i, j, k, l=0,\\ j\neq l}}b^r_ib^s_jb^r_kb^s_l\tr(P_iP_jP_lP_k)
+\sum^{m}_{i, j, k=0}b^r_i(b^s_j)^2b^r_k\tr(P_iP_k)\\
&=\sum^{m}_{\substack{i, j, k, l=0,\\ j<l}}b^r_ib^s_jb^r_kb^s_l\tr(P_i(P_jP_l+P_lP_j)P_k)
+\sum^{m}_{i, j=0}(b^r_i)^2(b^s_j)^2\tr(I_{2l})
=2l\sum^{m}_{i, j=0}(b^r_ib^s_j)^2,
\end{aligned}$$
$$\Delta_2:=\sum^{m}_{\substack{i, k, l=0,\\ k\neq l}}b^r_ib^s_ib^r_kb^s_l\tr(P_iP_iP_lP_k)
=\sum^{m}_{\substack{i, k, l=0,\\ k\neq l}}b^r_ib^s_ib^r_kb^s_l\tr(P_lP_k)=0, $$
$$\Delta_3:=\sum^{m}_{\substack{i, j, k=0,\\ i\neq j}}b^r_ib^s_jb^r_kb^s_k\tr(P_iP_jP_kP_k)
=\sum^{m}_{\substack{i, j, k=0,\\ i\neq j}}b^r_ib^s_jb^r_kb^s_k\tr(P_iP_j)=0, $$
and
$$\Delta_4:=\sum^{m}_{i, k=0}b^r_ib^s_ib^r_kb^s_k\tr(P_iP_iP_kP_k)
=\sum^{m}_{i, k=0}b^r_ib^s_ib^r_kb^s_k\tr(I_{2l})=2l\(\sum^{m}_{i=0}b^r_ib^s_i\)^2. $$
Therefore, by
\begin{equation}\label{sumdec}
\sum^{}_{\substack{i, j, k, l}}=
\sum^{}_{\substack{i, j, k, l,\\i=j,  k=l}}+
\sum^{}_{\substack{i, j, k, l,\\i=j,  k\neq l}}+
\sum^{}_{\substack{i, j, k, l,\\i\neq j,k= l}}+
\sum^{}_{\substack{i, j, k, l,\\i\neq j,k\neq l}},
\end{equation}
we have
\begin{equation}\label{CsDDVV3}
\|\[B_r,B_s\]\|^2
=4(\Delta_1-\Delta_2-\Delta_3-\Delta_4)
=8l\(\sum^{m}_{i, j=0}(b^r_ib^s_j)^2-\(\sum^{m}_{i=0}b^r_ib^s_i\)^2\).
\end{equation}
Still by (\ref{CsDDVV2}), we have
\begin{equation}\label{CsDDVV4}
\|B_r\|^2=\sum^{m}_{i=0}(b^r_i)^2\|P_i\|^2=\sum^{m}_{i=0}(b^r_i)^2\tr(I_{2l})=2l\sum^{m}_{i=0}(b^r_i)^2.
\end{equation}
Combining (\ref{CsDDVV3}) and (\ref{CsDDVV4}), the inequality (\ref{CsDDVV}) is transformed into the following:
$$\sum^M_{r, s=1}\sum^{m}_{i, j=0}(b^r_ib^s_j)^2
-\sum^M_{r, s=1}\(\sum^{m}_{i=0}b^r_ib^s_i\)^2
\leq \(1-\frac1N\)\(\sum^M_{r=1}\sum^{m}_{i=0}(b^r_i)^2\)^2, $$
i.e.,
$$\sum^M_{r, s=1}\(\sum^{m}_{i=0}b^r_ib^s_i\)^2
\geq \frac1N\(\sum^M_{r=1}\sum^{m}_{i=0}(b^r_i)^2\)^2,$$
which is implied by Lemma \ref{Clifford} and so is the equality condition.

The proof is complete.
\end{proof}

Analogously we are able to obtain the DDVV-type inequality for Clifford algebra.
\begin{thm}\label{TCaDDVV}
Let $\{E_1, \cdots, E_{m-1}\}$ be a Clifford algebra on $\mathbb{R}^l$, i.e.,  $E_1,\cdots,E_{m-1}\in O(l)$ are real skew-symmetric orthogonal matrices satisfying $E_iE_j+E_jE_i=-2\delta_{ij}I_{l}$. Let $B_1,\cdots,B_M \in span\{E_1, \cdots, E_{m-1}\}$, then
\begin{equation}\label{CaDDVV}
\sum^M_{r,s=1}\|\[B_r,B_s\]\|^2\leq \frac4l\(1-\frac1N\)\(\sum^M_{r=1}\|B_r\|^2\)^2,\quad
N=\min\{m-1, M\}.
\end{equation}
The condition for equality has two cases.

Case $(1)$: When $m-1\leq M$, the equality holds if and only if
$e_1, \cdots, e_{m-1}$ are orthogonal with the same norm,
where $e_i:=\(\<E_i, B_1\>, \cdots, \<E_i, B_M\>\)\in\mathbb{R}^M$.

Case $(2)$: When $m-1\geq M$, the equality holds if and only if
$B_1,\cdots,B_M$ are orthogonal with the same norm.

\end{thm}
\begin{proof}
For $1\leq r \leq M$, let $B_r=\sum^{m-1}_{i=1}b^r_iE_i$,
where $b^r_1, \cdots, b^r_{m-1}$ are real numbers.
Since
\begin{equation}\label{CaDDVV2}
E_iE_j+E_jE_i=-2\delta_{ij}I_{l}, \quad i,j=1,\cdots,m-1,
\end{equation}
we have
$$\[B_r,B_s\]=\sum^{m-1}_{i,j=1, i\neq j}b^r_ib^s_j\[E_i,E_j\]
=2\sum^{m-1}_{i,j=1, i\neq j}b^r_ib^s_jE_iE_j, $$
$$\|\[B_r,B_s\]\|^2=4\left\|\sum^{m-1}_{i,j=1, i\neq j}b^r_ib^s_jE_iE_j\right\|^2
=4\sum^{m-1}_{\substack{i, j, k, l=1,\\ i\neq j, k\neq l}}b^r_ib^s_jb^r_kb^s_l\tr(E_iE_jE_lE_k). $$
By (\ref{CaDDVV2}) and $\tr(E_iE_j)=\tr(E_jE_i)$ for all $1\leq i,j\leq m-1$, we have
$$\begin{aligned}
\Delta_1:&=\sum^{m-1}_{i, j, k, l=1}b^r_ib^s_jb^r_kb^s_l\tr(E_iE_jE_lE_k)\\
&=\sum^{m-1}_{\substack{i, j, k, l=1,\\ j\neq l}}b^r_ib^s_jb^r_kb^s_l\tr(E_iE_jE_lE_k)
+\sum^{m-1}_{i, j, k=1}b^r_i(b^s_j)^2b^r_k\tr(-E_iE_k)\\
&=\sum^{m-1}_{\substack{i, j, k, l=1,\\ j<l}}b^r_ib^s_jb^r_kb^s_l\tr(E_i(E_jE_l+E_lE_j)E_k)
+\sum^{m-1}_{i, j=1}(b^r_i)^2(b^s_j)^2\tr(I_{l})
=l\sum^{m-1}_{i, j=1}(b^r_ib^s_j)^2,
\end{aligned}$$
$$\Delta_2:=\sum^{m-1}_{\substack{i, k, l=1,\\ k\neq l}}b^r_ib^s_ib^r_kb^s_l\tr(E_iE_iE_lE_k)
=\sum^{m-1}_{\substack{i, k, l=1,\\ k\neq l}}b^r_ib^s_ib^r_kb^s_l\tr(-E_lE_k)=0, $$
$$\Delta_3:=\sum^{m-1}_{\substack{i, j, k=1,\\ i\neq j}}b^r_ib^s_jb^r_kb^s_k\tr(E_iE_jE_kE_k)
=\sum^{m-1}_{\substack{i, j, k=1,\\ i\neq j}}b^r_ib^s_jb^r_kb^s_k\tr(-E_iE_j)=0, $$
and
$$\Delta_4:=\sum^{m-1}_{i, k=1}b^r_ib^s_ib^r_kb^s_k\tr(E_iE_iE_kE_k)
=\sum^{m-1}_{i, j=1}b^r_ib^s_ib^r_jb^s_j\tr(I_{l})=l\(\sum^{m-1}_{i=1}b^r_ib^s_i\)^2. $$
Therefore, by the decomposition (\ref{sumdec}) of the sum about indices,
we have
\begin{equation}\label{CaDDVV3}
\|\[B_r,B_s\]\|^2
=4(\Delta_1-\Delta_2-\Delta_3-\Delta_4)
=4l\(\sum^{m-1}_{i, j=1}(b^r_ib^s_j)^2-\(\sum^{m-1}_{i=1}b^r_ib^s_i\)^2\).
\end{equation}
Still by (\ref{CaDDVV2}), we have
\begin{equation}\label{CaDDVV4}
\|B_r\|^2=\sum^{m-1}_{i=1}(b^r_i)^2\|E_i\|^2=l\sum^{m}_{i=0}(b^r_i)^2.
\end{equation}
Combining (\ref{CaDDVV3}) and (\ref{CaDDVV4}), the inequality (\ref{CaDDVV}) is transformed into the following:
$$\sum^M_{r, s=1}\sum^{m-1}_{i, j=1}(b^r_ib^s_j)^2
-\sum^M_{r, s=1}\(\sum^{m-1}_{i=1}b^r_ib^s_i\)^2
\leq \(1-\frac1N\)\(\sum^M_{r=1}\sum^{m-1}_{i=1}(b^r_i)^2\)^2, $$
i.e.,
$$\sum^M_{r, s=1}\(\sum^{m-1}_{i=1}b^r_ib^s_i\)^2
\geq \frac1N\(\sum^M_{r=1}\sum^{m-1}_{i=1}(b^r_i)^2\)^2,$$
which is implied by Lemma \ref{Clifford} and so is the equality condition.

The proof is complete.
\end{proof}

\begin{ack}
The authors would like to thank Professors Zhiqin Lu and David Wenzel for their very detailed and valuable comments.
\end{ack}


\end{document}